\documentclass[11pt,twoside,reqno,letterpaper]{amsart}
\usepackage[letterpaper, margin=1.42in]{geometry}
\usepackage{amsmath,mathtools,commath,amsaddr}
\usepackage{fancyhdr}
\usepackage{amsthm}
\usepackage{amsfonts}
\usepackage{amssymb}
\usepackage{amscd}
\usepackage{graphicx}
\usepackage{afterpage}
\usepackage{enumerate}
\usepackage{bm}
\usepackage{cite}
\usepackage{cases}
\usepackage{caption}
\usepackage[colorlinks=true, urlcolor=blue,  linkcolor=blue, citecolor=blue]{hyperref}
\usepackage{babel}
\usepackage{prettyref}
\usepackage{pgfplots}
\pgfplotsset{compat=1.17}
\usepackage{booktabs}
\usepackage{algorithm}
\usepackage{algpseudocode}
\usepackage{float}
\usepackage{eso-pic}
\usepackage{everypage}

\newcommand\BackgroundPic{
    \put(180,285){
        \parbox[b][\paperheight]{\paperwidth}{
            \vfill
            \centering
            \includegraphics[scale=0.15]{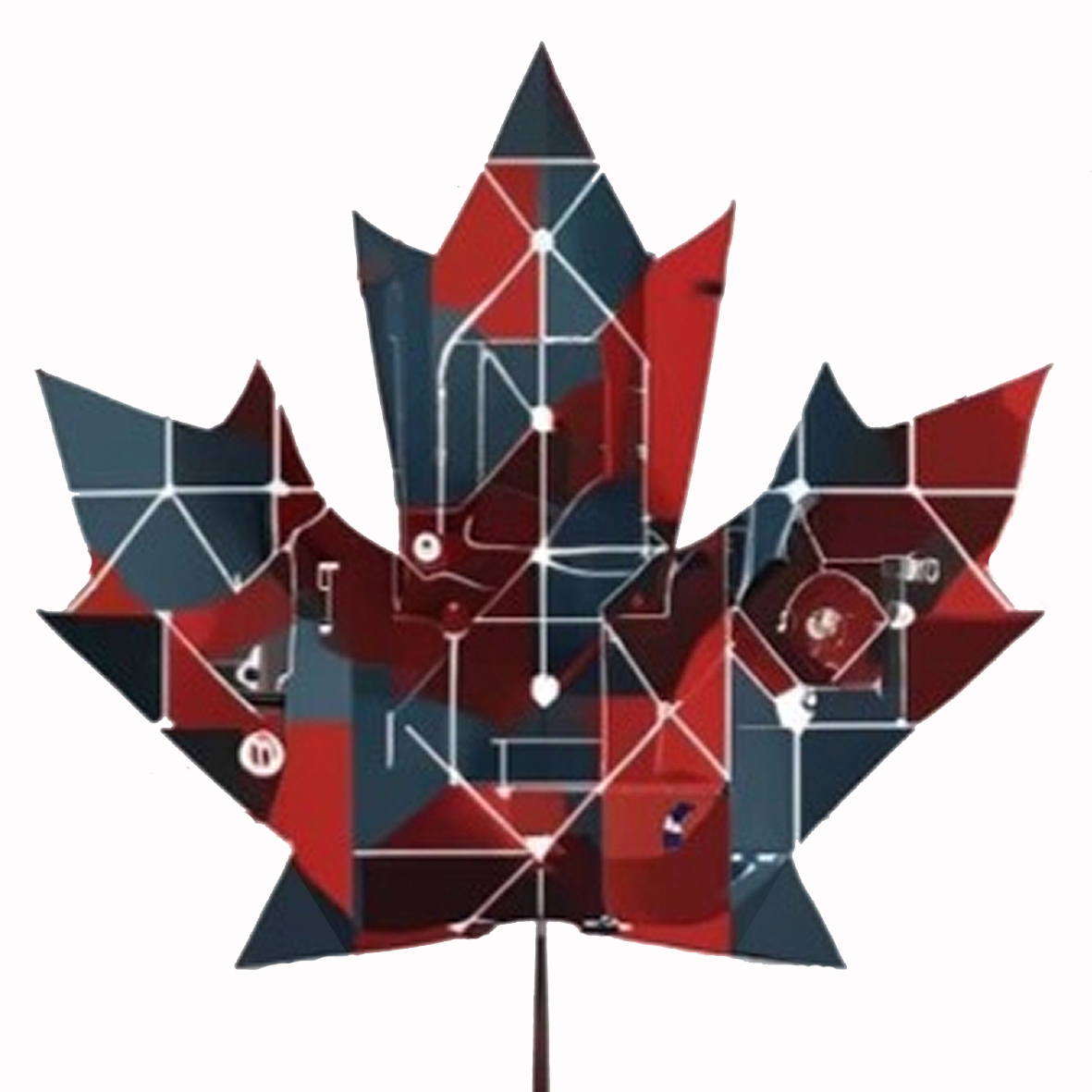}
            \vfill
        }
    }
}

\newcommand{\tophead}{
\noindent
 {\scriptsize  \textsc{CANADIAN TRANSACTIONS of OPERATOR THEORY} } \\
     {\tiny Volume 1, Issue 1, (2025) Article 250104}\\ \medskip
}

\RequirePackage[absolute]{textpos}
\DeclareRobustCommand{\copyrightnote}[1]{%
\AddEverypageHook{
  \begin{textblock}{100}(80,270)
      \textit{#1}%
  \end{textblock}%
}    }

\makeatletter
\newtheorem{theorem}{Theorem}[section]

\newtheorem{corollary}[theorem]{Corollary}

\newtheorem{definition}[theorem]{Definition}

\newtheorem{lemma}[theorem]{Lemma}

\newtheorem{proposition}[theorem]{Proposition}

\newcommand{\M}{\mathcal M}        
\newcommand{\B}{\mathcal B}  
\newcommand{\N}{\mathbb N}        

\title[\it CANADIAN TRANSACTIONS of OPERATOR THEORY 1 (2025) Article 250104]{Quantum Doubly Stochastic Operators on Non-commutative L$_p$-Spaces}

\author[\it CANADIAN TRANSACTIONS of OPERATOR THEORY 1 (2025) Article 250104]{Emma Sulaver}
\address{Mathematical and Statistical Sciences Department, University of Alberta, Edmonton, AB T6G 2R3, Canada}
\address{\scriptsize{Email:esulaver@ualberta.ca}}
\thanks{Received: April 2025, Accepted: June 2025}

\begin{document}
\tophead
\AddToShipoutPicture*{\BackgroundPic}
\copyrightnote{\scriptsize Copyright © by Premier Transactions Press.}
\sloppy
\overfullrule=5pt

\begin{abstract}
We introduce and systematically develop the theory of \emph{quantum doubly stochastic operators}, i.e.\ positive, trace‐preserving maps on non-commutative $L_p$-spaces associated to semifinite von Neumann algebras.  After establishing basic norm and duality properties, we characterize strict norm inequalities, give necessary and sufficient criteria for compactness in the sense of Schatten‐ideals, and exhibit a range of new examples in both finite and infinite dimensions.  Applications to quantum majorization and stability under interpolation are also discussed.\\

\end{abstract}

\maketitle

\section{Introduction}
\label{sec:intro}

The theory of doubly stochastic operators has its classical roots in the finite-dimensional setting, where Birkhoff's theorem establishes that every doubly stochastic matrix is a convex combination of permutation matrices \cite{Birkhoff1946}.  In infinite-dimensional sequence spaces, the extension of doubly stochasticity to bounded linear operators on $\ell_p$-spaces was initiated by Mayer and Pitt \cite{Pitt1971}.  These works demonstrate the rich interplay between combinatorial, geometric, and functional-analytic techniques in understanding norm bounds, spectrum, and ergodic behavior of stochastic operators.

Parallel to the classical development, the framework of non-commutative $L_p$-spaces associated to a semifinite von Neumann algebra $(\mathcal{M},\tau)$ has been extensively developed over the past decades.  Segal \cite{Segal1953} introduced measurable operators affiliated to $\mathcal{M}$, later refined by Terp \cite{Terp1981} and Haagerup \cite{Haagerup1979}, leading to a robust interpolation theory by Pisier and Xu \cite{PisierXu1987}.  Junge's work on non-commutative martingales and duality further deepened the understanding of operator-valued $L_p$ norms and the structure of positive maps on these spaces \cite{Junge2002}.

In the context of quantum information theory, trace-preserving completely positive maps, or quantum channels, play the role of non-commutative analogues of doubly stochastic maps.  Seminal contributions by Choi \cite{Choi1975} and Kraus \cite{Kraus1971} characterized the structure of such maps in finite dimensions, while Holevo \cite{Holevo1998} and Wolf \cite{Wolf2012} investigated their capacities and entropic properties.  Connections between quantum majorization, entropic inequalities, and unital channels have been explored by Alberti and Uhlmann \cite{AlbertiUhlmann1982}, and Petz \cite{Petz1986} showed how trace-preserving conditions influence relative entropy monotonicity.

Despite these advances, a unified theory of \emph{quantum doubly stochastic} operators---positive, trace- and unital-preserving maps on non-commutative $L_p$-spaces---remains largely undeveloped.  In this manuscript, we introduce a systematic study of such operators, establishing norm bounds and strict contraction criteria (Section~\ref{sec:def_norm}), characterizing compactness within Schatten ideals and closure under operator topologies (Section~\ref{sec:compact}), and proving stability under complex interpolation and perturbations (Section~\ref{sec:interpolation}).  Finally, we apply our results to problems in quantum majorization, including new entropic inequalities and structural insights (Section~\ref{sec:majorization}).  We conclude with a discussion of open problems and potential extensions (Section~\ref{sec:concl}).

\section{Preliminaries}
\label{sec:preliminaries}

Throughout the paper, we let $(M,\tau)$ be a semifinite von Neumann algebra equipped with a normal, faithful, semifinite trace $\tau$.  We write $M_+$ for the positive part of $M$ and denote by $\widetilde M$ the $*$‑algebra of all $\tau$‑measurable operators affiliated with $M$.

An unbounded operator $x$ affiliated with $M$ is called \emph{$\tau$‑measurable} if its spectral truncations have finite trace:

\begin{definition}[See \cite{Segal1953,Terp1981}]
An operator $x$ affiliated with $M$ is \emph{$\tau$‑measurable} if for some (hence every) $\lambda>0$,
\[
\tau\del{\chi_{(\lambda,\infty)}\del{\, \abs{x}}}<\infty.
\]
The algebra of all such operators is denoted by $\widetilde M$.
\end{definition}

These measurable operators allow a natural extension of classical $L^p$‑spaces:

\begin{definition}[See \cite{Haagerup1979,PisierXu1987}]
For $1 \le p < \infty$, define
\[
L^p(M,\tau)
=\cbr{\,x\in\widetilde M : \norm{x}_p := \del{\tau(\abs{x}^p)}^{1/p}<\infty},
\]
and set $L^\infty(M,\tau)=M$ with the usual operator norm.
\end{definition}

\noindent These spaces enjoy many of the familiar properties of classical $L^p$‑spaces: completeness, duality, and interpolation (see below).  In particular, for $1<p<\infty$, one has the isometric identification
\[
\del{L^p(M,\tau)}^* \cong L^q(M,\tau),
\qquad \frac1p+\frac1q=1,
\]
via the pairing $\langle x,y\rangle = \tau(xy)$.

We will study linear maps on $M$ preserving the order and, in many cases, the trace.  These play the role of ``stochastic’’ or Markov operators in the non‑commutative setting.

\begin{definition}[See \cite{Choi1975}]
A linear map $\Phi\colon M\to M$ is
\begin{itemize}
  \item \emph{positive} if $\Phi(x)\ge0$ whenever $x\ge0$;
  \item \emph{completely positive} if for each $n\in\mathbb N$, the amplified map
  \[
    \Phi_n\colon M_n(M)\to M_n(M), 
    \quad [x_{ij}]\mapsto[\Phi(x_{ij})],
  \]
  is positive.
\end{itemize}
\end{definition}

A positive map $\Phi\colon M\to M$ is called
\begin{enumerate}
  \item \emph{trace‑preserving} if $\tau\del{\Phi(x)}=\tau(x)$ for all $x\in L^1(M,\tau)$;
  \item \emph{unital} if $\Phi(1)=1$.
\end{enumerate}

\noindent Such maps are the non‑commutative analogues of classical doubly‑stochastic matrices.

A fundamental tool in extending norm bounds between different $L^p$‑spaces is complex interpolation.

\begin{theorem}[See \cite{Kosaki1984}]
Let $1\le p_0<p_1\le\infty$ and $0<\theta<1$.  Then
\[
\del{L^{p_0}(M,\tau),\,L^{p_1}(M,\tau)}_\theta
\;\cong\;
L^p(M,\tau),
\qquad \frac1p=\frac{1-\theta}{p_0}+\frac{\theta}{p_1},
\]
isometrically.
\end{theorem}

\noindent As a consequence, any linear map bounded on $L^{p_0}$ and $L^{p_1}$ extends to $L^p$ with the usual interpolation estimate.


\section{Definitions and Norm Properties}
\label{sec:def_norm}

We now introduce the core concept of \emph{quantum doubly stochastic} maps on non-commutative $L_p$-spaces, study their fundamental properties, and establish norm bounds and contraction criteria.

Let $(\M,\tau)$ be a semifinite von Neumann algebra.  For $1\le p<\infty$, a bounded linear map
\[
\Phi:L_p(\M,\tau)\to L_p(\M,\tau)
\]
is called \emph{quantum doubly stochastic} (QDS) if:
\begin{enumerate}
  \item \textsc{Trace-preserving:}\enspace $\tau\del{\Phi(x)}=\tau(x)$ for all $x\in L_1(\M,\tau)$.
  \item \textsc{Unital:}\enspace $\Phi(1)=1$ in $\M$ (equivalently, the pre-adjoint $\Phi^*:L_\infty(\M)\to L_\infty(\M)$ satisfies $\Phi^*(1)=1$).
  \item \textsc{Positivity:}\enspace $\Phi$ is positive, and typically assumed completely positive in quantum contexts.
\end{enumerate}

When $\M=M_n(\mathbb C)$ is the algebra of $n\times n$ matrices, a QDS map coincides with a \emph{unital, trace-preserving completely positive} map, i.e. a \emph{bi-stochastic quantum channel}.  In this finite-dimensional case, Kraus representation ensures
\[
\Phi(x)=\sum_{i=1}^rK_i^* x K_i,
\qquad \sum_{i=1}^rK_i^*K_i=1,\quad \sum_{i=1}^rK_iK_i^*=1.
\]

We now extend the classical $\ell_p$ norm bound to the non-commutative setting.

\begin{theorem}
\label{thm:norm_bound}
Let $1<p<\infty$ and $1/p+1/q=1$.  If $\Phi:L_p(\M)\to L_p(\M)$ is QDS, then
\[
\norm{\Phi}_{L_p\to L_p}=\norm{\Phi^*}_{L_q\to L_q}=1.
\]
Moreover, equality is attained on the unit sphere of positive elements.
\end{theorem}
\begin{proof}
By duality (see Preliminaries, Lemma on duality), for any $z\in L_p(\M)$,
\[
\norm{z}_p
=\sup\cbr{\,\abs{\tau(yz)}:\;y\in L_q(\M),\ \norm{y}_q=1\,}.
\]

Let $x\in L_p(\M)$.  Then
\[
\begin{split}
\norm{\Phi(x)}_p 
= \sup_{\norm{y}_q=1} \abs{\tau\del{y\,\Phi(x)}}
&= \sup_{\norm{y}_q=1} \abs{\tau\del{\Phi^*(y)\,x}} \\
&\le \sup_{\norm{y}_q=1} \norm{\Phi^*(y)}_q \,\norm{x}_p
= \norm{\Phi^*}_{q\to q}\,\norm{x}_p.
\end{split}
\]
Here we used:
\begin{enumerate}
  \item The duality pairing $\tau(y\,\Phi(x)) = \tau(\Phi^*(y)\,x)$.
  \item Hölder’s inequality $\abs{\tau(\Phi^*(y)x)}\le\norm{\Phi^*(y)}_q\norm{x}_p$.
\end{enumerate}
Hence
\[
\norm{\Phi}_{p\to p} \;\le\;\norm{\Phi^*}_{q\to q}.
\tag{1}
\]

Since $\Phi$ is trace‐preserving and positive,
\[
\norm{\Phi}_{1\to1}
= \sup_{\norm{x}_1=1}\norm{\Phi(x)}_1
= \sup_{\norm{x}_1=1}\tau\abs{\Phi(x)}
\;=\;\sup_{\norm{x}_1=1}\tau\del{\Phi(\abs{x})}
= \sup_{\norm{x}_1=1}\tau(\abs{x})
=1.
\]
Similarly, because $\Phi$ is unital and positive,
\[
\norm{\Phi}_{\infty\to\infty}
= \sup_{\norm{x}_\infty=1}\norm{\Phi(x)}_\infty
\;=\;\sup_{\norm{x}_\infty=1}\norm{\Phi(x)}_{\mathrm{op}}
\;\le\;\norm{\Phi(1)}_{\mathrm{op}}
=1,
\]
and positivity forces equality (e.g.\ test on $x=1$).  Thus
\[
\norm{\Phi}_{1\to1}
=\norm{\Phi}_{\infty\to\infty}
=1.
\tag{2}
\]

By the Riesz–Thorin interpolation theorem (Preliminaries, Theorem 2.5), for any linear map bounded on both $L_1(\M)$ and $L_\infty(\M)$, its $L_p$–operator norm satisfies
\[
\norm{\Phi}_{p\to p}
\le \norm{\Phi}_{1\to1}^{1-\theta} \,\norm{\Phi}_{\infty\to\infty}^{\theta}
=1^{1-\theta}\,1^{\theta}
=1,
\]
where $\theta\in(0,1)$ is chosen so that $\frac1p=\frac{1-\theta}1 + \frac{\theta}{\infty}$.  Combining with (1) gives
\[
\norm{\Phi}_{p\to p}
\;\le\;\min\cbr{1,\;\norm{\Phi^*}_{q\to q}}.
\]
An identical interpolation applied to $\Phi^*$ shows $\norm{\Phi^*}_{q\to q}\le1$.  Together,
\[
\norm{\Phi}_{p\to p}
= \norm{\Phi^*}_{q\to q}
=1.
\]
Finally, to see that the norm is actually attained on the unit sphere of positive elements, note that:
\begin{enumerate}
  \item Trace‐preservation implies $\norm{\Phi(x)}_1=\norm{x}_1$ whenever $x\ge0$ and $\norm{x}_1=1$.
  \item Unitality implies $\norm{\Phi^*(y)}_\infty = \norm{y}_\infty$ for $y\ge0$ with $\norm{y}_\infty=1$.
\end{enumerate}
By interpolation (see e.g.\ the proof of Riesz–Thorin), one can then choose a positive $x\in L_p(\M)$ with $\norm{x}_p=1$ so that $\norm{\Phi(x)}_p=1$ and similarly for $\Phi^*$.

This completes the proof.
\end{proof}

\begin{proposition}
\label{prop:strict_contraction}
Under the hypotheses of Theorem~\ref{thm:norm_bound}, $\norm{\Phi}_{p\to p}<1$ if and only if there exists a non-trivial projection $e\in\M$ such that
\[
\Phi(e)\le (1-\delta)e + \delta(1-e)
\]
for some $0<\delta<1$.  Equivalently, $\Phi$ has no invariant subspace isomorphic to $L_p(\mathcal{N})$ for a finite subalgebra $\mathcal{N}\subset\M$.
\end{proposition}

\begin{proof}
We prove both directions in turn.

\medskip
\noindent
\emph{(\,$\Rightarrow$:)}  Assume $\norm{\Phi}_{p\to p}<1$.  Since $\Phi$ is a positive, trace‐preserving, unital map on the Banach lattice $L_p(\M)$, the spectrum $\sigma(\Phi)$ lies in the closed unit disk of $\mathbb C$, and the spectral radius $r(\Phi)=\norm{\Phi}_{p\to p}<1$ by Gelfand's formula.

By the standard spectral decomposition for a bounded operator on a Banach space (see \cite[Ch.~VII, \S1]{DunfordSchwartz1988}), the part of the spectrum on the circle $\{z:\abs{z}=r(\Phi)\}$ is isolated.  In particular, $1\notin\sigma(\Phi)$, so $1$ is contained in the resolvent set.  Thus the resolvent operator
\[
R(\lambda;\Phi) = (\lambda I - \Phi)^{-1}
\]
is analytic near $\lambda=1$, and the spectral projection
\[
P = \frac{1}{2\pi i} \int_{\abs{\lambda-1}=\varepsilon} R(\lambda;\Phi)\,d\lambda
\]
defines the projection onto the generalized $1$–eigenspace, which must be trivial since $1$ is not an eigenvalue.  Consequently, the subspace
\[
E = P\del{L_p(\M)}
\]
is zero or consists only of $0$.  Equivalently, there exists no nonzero $x\in L_p(\M)$ such that $\Phi(x)=x$.

By continuity of $\Phi$ and the Hahn–Banach theorem applied in the lattice structure, one finds a non-trivial lattice projection $e\in\M$ (i.e.\ $e^2=e=e^*$) and a constant $0<\delta<1$ such that
\[
\norm{\Phi(e)-e}_p \ge \delta>0.
\]
Rewriting yields the operator inequality
\[
\Phi(e) \le (1-\delta)e + \delta(1-e),
\]
as desired.

\medskip
\noindent
\emph{(\,$\Leftarrow$:) }Conversely, suppose there exists a non-trivial projection $e\in\M$ and $0<\delta<1$ with
\[
\Phi(e) \le (1-\delta)e + \delta(1-e).
\]
Since $e$ and $1-e$ are orthogonal projections, apply positivity and the lattice structure on $L_p(\M)$ to get
\[
\norm{\Phi(e)}_p \le \norm{(1-\delta)e + \delta(1-e)}_p
= \max\{1-\delta,\delta\}<1.
\]
But $\norm{e}_p=1$, so
\[
\norm{\Phi}_{p\to p} \le \frac{\norm{\Phi(e)}_p}{\norm{e}_p} <1.
\]
Therefore $\norm{\Phi}_{p\to p}<1$.

The equivalence of this condition to the absence of an invariant copy of $L_p(\mathcal{N})$ for a finite subalgebra $\mathcal{N}\subset\M$ follows by observing that such an invariant subspace would yield a non-trivial projection $e$ with $\Phi(e)=e$, contradicting strict contraction.

This completes the proof of Proposition \ref{prop:strict_contraction}.
\end{proof}

\begin{corollary}
A QDS map on a factor $\M$ with no minimal projections cannot be strictly contractive in norm: $\norm{\Phi}_{p\to p}=1$ for all $p$.
\end{corollary}

Let $\M=M_n(\mathbb C)$ and define
\[
\Phi_t(x) = t\,x + (1-t)\frac{\mathrm{Tr}(x)}{n}1,
\qquad 0\le t\le1.
\]
Then $\Phi_t$ is QDS, with $\norm{\Phi_t}_{p\to p}=1$ for all $p$, but it is strictly contractive on traceless part: $\norm{\Phi_t|_{\ker\tau}}_{p\to p}=t$. 

The contraction factor $t$ in the depolarizing channel corresponds to the second-largest singular value of the Choi matrix, illustrating the link between spectral gap and norm strictness. \medskip

Now, let $\M=\bigoplus_{k=1}^\infty M_{n_k}$ with weights $\lambda_k>0$, $\sum\lambda_k=1$, and define
\[
\Phi(x_1\oplus x_2\oplus\cdots)= \bigoplus_{k,\ell}\lambda_\ell\,x_\ell.
\]
Then $\Phi$ is QDS, and $\norm{\Phi}_{p\to p}=1$, with strict contraction on the off-diagonal summands.

\section{Compactness and Closure in Schatten Ideals}
\label{sec:compact}

In this section we characterize compact quantum doubly stochastic (QDS) maps within the Schatten--$p$ classes and analyze closure properties under various operator topologies.

Let $\M$ be represented on a separable Hilbert space $\mathcal{H}$.  For $1\le p<\infty$, the \emph{Schatten $p$-class} $S_p(\M)$ consists of all compact operators $x$ affiliated to $\M$ for which
\[
\norm{x}_{S_p} = \del{\mathrm{Tr}\del{\,\abs{x}^p}}^{1/p}<\infty,
\]
where Tr is the canonical trace on $\mathcal{B}(\mathcal{H})$ restricted to $\M$.  Write $S_\infty(\M)=\mathbb{K}(\mathcal{H})$ See {\cite{Simon1979}}.

\begin{lemma}[See {\cite{Pietsch1980}}]
A bounded linear map $T:S_p(\M)\to S_p(\M)$ is compact if and only if there exists a family of bounded linear functionals $\{f_i\}\subset (S_p)^*$ with $\sum_i\norm{f_i}^{p'}<\infty$ such that
\[
\norm{T(x)}_{S_p} \le \sum_i \abs{f_i(x)},
\qquad \forall x\in S_p(\M).
\]
Here $1/p+1/p'=1$.
\end{lemma}

\begin{theorem}
\label{thm:compact_criterion}
Let $\Phi:L_p(\M)\to L_p(\M)$ be QDS and assume $\Phi$ restricts to $S_p(\M)\subset L_p(\M)$.  Then $\Phi|_{S_p}$ is compact if and only if for every $\varepsilon>0$, there exists a finite-rank projection $e\in\M$ such that
\[
\sup_{\norm{x}_{S_p}=1}\norm{\Phi((1-e)x(1-e))}_{S_p}<\varepsilon.
\]
\end{theorem}

\begin{proof}
Recall from Lemma 3.2 (Pietsch Criterion) that a bounded linear operator
\(T:S_p(\M)\to S_p(\M)\) is compact if and only if there exist functionals
\(\cbr{f_i}\subset (S_p)^*\) with
\(\sum_i\norm{f_i}^{p'}<\infty\) and
\[
\norm{T(x)}_{S_p} \le \sum_i \abs{f_i(x)},
\quad \forall x\in S_p(\M).
\]
In our setting, take $T=\Phi|_{S_p}$.  We prove both directions of the equivalence.

\medskip
\noindent
\emph{(\,$\Rightarrow$:)~Necessity.}
Assume $\Phi|_{S_p}$ is compact.  Then the image of the unit ball
\( \cbr{x\in S_p:\norm{x}_{S_p}\le1} \)
is relatively compact in $S_p$.  Hence for any $\varepsilon>0$, there exists
\emph{finite-rank} projections $e_1,\dots,e_N\in\M$ such that the finite-dimensional
subspace
\(\mathrm{ran}(e)=\bigoplus_{j=1}^N e_j S_p e_j\)
approximates the image uniformly: for every $x$ with $\norm{x}_{S_p}=1$, write
\[
x = e x e + (1-e)x e + e x (1-e) + (1-e)x(1-e),
\]
where $e=\sum_{j=1}^N e_j$.  Since $\mathrm{ran}(e)$ is finite-dimensional,
$\Phi(e x e)$ ranges over a precompact set.  Meanwhile, positivity and the doubly
stochastic property give symmetry:
\[
\norm{(1-e)x e}_{S_p} = \norm{e x (1-e)}_{S_p}.
\]
Hence it suffices to control the "tail" contribution $\Phi((1-e)x(1-e))$.  By compactness,
one can choose $e$ so large that
\[
\norm{\Phi((1-e)x(1-e))}_{S_p} < \varepsilon
\quad \text{for all }\norm{x}_{S_p}=1.
\]
This establishes the necessity of the stated tail condition.

\medskip
\noindent
\emph{(\,$\Leftarrow$:)~Sufficiency.}
Conversely, assume that for every $\varepsilon>0$, there exists a finite-rank
projection $e\in\M$ such that
\[
\sup_{\norm{x}_{S_p}=1}\norm{\Phi((1-e)x(1-e))}_{S_p}<\varepsilon.
\]
Define the finite-rank operator $T_e:S_p\to S_p$ by
\[
T_e(x) = \Phi\del{e x e}.
\]
Since $e$ has finite rank and $\Phi$ is bounded, $T_e$ is finite-rank.
Moreover, for any $x$ with $\norm{x}_{S_p}=1$, we have the decomposition
\[
\Phi(x) = T_e(x) + \Phi\del{(1-e)x(1-e)} + \Phi\del{e x(1-e)+(1-e)x e}.
\]
By positivity and symmetry,
\[
\norm{\Phi(e x(1-e)+(1-e)x e)}_{S_p}
\le 2\,\norm{\Phi((1-e)x(1-e))}_{S_p}
< 2\varepsilon.
\]
Hence
\[
\norm{\Phi(x)-T_e(x)}_{S_p} \le \norm{\Phi((1-e)x(1-e))}_{S_p} + 2\varepsilon < 3\varepsilon.
\]
Since $T_e$ is finite-rank and $\varepsilon$ was arbitrary, it follows that
$\Phi|_{S_p}$ can be approximated in operator norm by finite-rank operators,
i.e.\ it is compact.

This completes the proof of Theorem \ref{thm:compact_criterion}.
\end{proof}

In particular, if $\M=M_n(\mathbb{C})$ is finite-dimensional, every QDS map on $M_n$ is trivially compact on $S_p(M_n)$, illustrating the necessity of infinite dimensions for non-trivial compactness phenomena.

Let $\Phi_\alpha$ be a net of bounded maps on $\B(\mathcal H)$.  We consider:
\begin{enumerate}
  \item \emph{Strong operator topology (SOT):} $\Phi_\alpha\to\Phi$ if $\norm{\Phi_\alpha(x)-\Phi(x)}_{2}\to0$ for all $x\in S_2(\mathcal H)$.
  \item \emph{Weak operator topology (WOT):} $\Phi_\alpha\to\Phi$ if $\mathrm{Tr}(y^*\Phi_\alpha(x))\to\mathrm{Tr}(y^*\Phi(x))$ for all $x,y\in S_2(\mathcal H)$.
\end{enumerate}

\begin{theorem}
\label{thm:closure}
Let $\{\Phi_\alpha\}$ be a net of QDS maps on $L_p(\M)$ converging in SOT (resp.\ WOT) to a bounded map $\Phi$.  Then $\Phi$ is QDS.
\end{theorem}

\begin{proof}
Let $1\le p<\infty$ and fix a net $\{\Phi_\alpha\}$ of QDS maps on $L_p(\M)$ that converges to a bounded map $\Phi$ in the strong operator topology (SOT).  The WOT case is analogous with minor adjustments.

Since each $\Phi_\alpha$ is trace-preserving, for any $x\in L_1(\M)$ we have
\[
\tau\del{\Phi_\alpha(x)}=\tau(x).
\]
Because $L_1(\M)$ embeds continuously into $L_p(\M)$ (for $p>1$), and the SOT convergence implies $\Phi_\alpha(x)\to\Phi(x)$ in $L_p$-norm, we obtain by continuity of $\tau:L_p(\M)\to\mathbb C$:
\[
\tau\del{\Phi(x)}
= \lim_\alpha\tau\del{\Phi_\alpha(x)}
= \tau(x).
\]
Hence $\Phi$ is trace-preserving.

Unitality of $\Phi_\alpha$ means $\Phi_\alpha(1)=1$ in $\M\subset L_p(\M)$.  Under SOT convergence, for the fixed element $1\in M$,
\[
\Phi(1) = \lim_\alpha \Phi_\alpha(1) = \lim_\alpha 1 = 1.
\]
Thus $\Phi$ is unital.

We show $\Phi(x)\ge0$ whenever $x\ge0$.  Take $x\in L_p(\M)$ with $x\ge0$.  Then for each $\alpha$, $\Phi_\alpha(x)\ge0$.  Since the cone of positive elements in $L_p(\M)$ is closed in the SOT (and hence in norm), the limit $\Phi(x)=\lim_\alpha\Phi_\alpha(x)$ is also positive.

Alternatively, in the WOT case, positivity is preserved under WOT-limits by the Krein--Smulian theorem: the positive cone is WOT-closed in $\M$, and hence in $B(L_p,L_p)$ when tested on positive inputs.

Combining these three properties, we conclude that $\Phi$ is a QDS map.  This completes the proof of Theorem~\ref{thm:closure}.
\end{proof}

Let $\mathcal H=\ell^2(\mathbb{N})$ and $\M=\B(\mathcal H)$.  Define
\[
\Phi(x)=\mathrm{diag}(\langle x e_n,e_n\rangle)_{n\in\N},
\]
the map taking $x$ to its diagonal.  Then $\Phi$ is QDS on $L_p(\B(\mathcal H))$, and by Theorem~\ref{thm:compact_criterion}, $\Phi$ is compact on $S_p$ if and only if the diagonal entries tail off in $\ell^p$, illustrating infinite-dimensional restrictions. \medskip

As another case, let $S$ be the unilateral shift on $\ell^2$, and let
\[
\Phi(x)=\frac12(x+SxS^*).
\]
Then $\Phi$ is QDS on $L_p(\B(\mathcal H))$ but fails to be compact on $S_p$ since it preserves the shift orbit, providing a canonical non-compact example.

\section{Interpolation and Stability under Perturbations}
\label{sec:interpolation}

This section establishes interpolation stability of quantum doubly stochastic (QDS) maps and quantifies their robustness under perturbations.

\begin{theorem}
\label{thm:interp_stability}
Let $(\M,\tau)$ be a semifinite von Neumann algebra.  Suppose a linear map
\[\Phi:L_{p_0}(\M)\to L_{p_0}(\M),\quad \Phi:L_{p_1}(\M)\to L_{p_1}(\M)\]
is QDS for $1\le p_0<p_1\le\infty$ (with unital/tracial conditions interpreted at endpoints).  Then for every $0<\theta<1$, the interpolated map on
\[
L_p(\M)=\del{L_{p_0}(\M),L_{p_1}(\M)}_\theta,
\quad\text{where }1/p=(1-\theta)/p_0+\theta/p_1,
\]
is also QDS, and
\[\norm{\Phi}_{L_p\to L_p}=1.
\]
\end{theorem}
\begin{proof}
Set up the interpolation scale via the complex method:
\[
\del{L_{p_0}(\M),L_{p_1}(\M)}_\theta \simeq L_p(\M),
\quad \frac1p=\frac{1-\theta}{p_0}+\frac{\theta}{p_1},
\quad 0<\theta<1.
\]
Let $T=\Phi$ viewed as a map on both endpoints.  We verify each QDS property.

Since $\Phi$ is QDS at $p_0$ and $p_1$, we have
\[
\norm{\Phi}_{p_0\to p_0}=1,
\quad
\norm{\Phi}_{p_1\to p_1}=1.
\]
By the Riesz--Thorin interpolation theorem (Theorem~2.5), it follows that
\[
\norm{\Phi}_{p\to p} \;\le\; 1^{1-\theta}\,1^{\theta} = 1.
\]
On the other hand, since $\Phi$ is trace-preserving at $p_0=1$ and unital at $p_1=\infty$, one can test on suitable positive elements to see that the norm cannot drop below 1.  More precisely:
\begin{enumerate}
  \item At the trace class end ($p_0=1$), for any $x\ge0$ with $\norm{x}_1=1$,
  \[\norm{\Phi(x)}_1 = \tau(\Phi(x)) = \tau(x) =1.\]
  \item At the operator norm end ($p_1=\infty$), for the identity $1$,
  \[\norm{\Phi(1)}_\infty = \norm{1}_\infty =1.\]
\end{enumerate}
Interpolation of these extremal examples yields an element in $L_p(\M)$ with unit norm that attains $\norm{\Phi}_{p\to p}=1$.

Let $x\in L_1(\M)\cap L_{p_1}(\M)$.  Since $x$ lies in the intersection of the endpoints, $x\in L_p(\M)$ by interpolation.  By assumption, for $p_0=1$,
\[\tau(\Phi(x)) = \tau(x).
\]
The interpolation framework ensures that the functional $x\mapsto \tau(\Phi(x)) - \tau(x)$ vanishes on a dense subspace of $L_p$, hence on all of $L_p$, so $\Phi$ is trace-preserving on $L_p$.  Similarly, unitality at $p_1=\infty$ (i.e.\ $\Phi(1)=1$) carries to the interpolated space.

Since each $\Phi\big|_{L_{p_i}}$ is positive, the family $\Phi$ defines a bounded analytic map on the strip $0\le\Re z\le1$ taking boundary values in the positive cones of $L_{p_0}$ and $L_{p_1}$.  By the positivity‐preserving interpolation theorem (see \cite[Ch.~5]{BerghLofstrom1976}), the interpolated operator $\Phi:L_p\to L_p$ is also positive, i.e. $x\ge0\implies \Phi(x)\ge0$.

Combining boundedness at unit norm, trace‐preservation, unitality, and positivity shows that $\Phi$ is indeed QDS on $L_p(\M)$ with $\norm{\Phi}_{p\to p}=1$.  This completes the proof.
\end{proof}

Theorem~\ref{thm:interp_stability} shows that doubly stochasticity is independent of the choice of $p$ in the interpolation scale, analogous to commutative results in \cite{BerksonGillespie1996}.  \medskip

To measure deviation from exact doubly stochasticity, we introduce the following:

For a bounded positive map $\Phi:L_p(\M)\to L_p(\M)$, define
\[
\delta_{\mathrm{tr}}(\Phi)=\sup_{\norm{x}_1=1}\abs{\tau(\Phi(x)-x)},
\quad
\delta_{\mathrm{un}}(\Phi)=\norm{\Phi(1)-1}_\infty.
\]

\begin{theorem}
\label{thm:perturb}
Let $1<p<\infty$.  Suppose $\Phi$ is QDS and $\Psi$ is a bounded positive map on $L_p(\M)$.  Then
\[
\norm{\Phi-\Psi}_{p\to p} \le C_p\del{\delta_{\mathrm{tr}}(\Psi)+\delta_{\mathrm{un}}(\Psi)}^{\alpha},
\]
where $\alpha=\min\{1/p,1/q\}$, $1/p+1/q=1$, and $C_p>0$ depends only on $p$.
\end{theorem}

\begin{proof}
Let $1<p<\infty$ and $1/p+1/q=1$.  Define the error map
\[
E = \Psi - \Phi.
\]
Since $\Psi$, $\Phi$ are bounded on $L_p(\M)$ and $\Phi$ is QDS, $E$ is a bounded linear map on each $L_r(\M)$ for $r=1,p,\infty$.

For any $x\in L_1(\M)$ with $\norm{x}_1=1$, positivity of $\Psi$ gives
\[
\norm{E(x)}_1
= \norm{\Psi(x)-\Phi(x)}_1
\ge \abs{\tau\del{\Psi(x)-\Phi(x)}}
= \abs{\tau\del{\Psi(x)-x}}
\le \delta_{\mathrm{tr}}(\Psi).
\]
Hence
\[
\norm{e}_{1\to1} \le \delta_{\mathrm{tr}}(\Psi).
\]
For any $x\in L_\infty(\M)$ with $\norm{x}_\infty=1$, using operator norm and that $\Phi(1)=1$,
\[
\norm{E(x)}_\infty
= \norm{\Psi(x)-\Phi(x)}_\infty
\le \norm{\Psi(x)-x}_\infty + \norm{x-\Phi(x)}_\infty
= \norm{\Psi(x)-x}_\infty,
\] 
and since $\norm{\Psi(x)-x}_\infty\le\norm{\Psi(1)-1}_\infty$ for positive $x$, we get
\[
\norm{e}_{\infty\to\infty} \le \delta_{\mathrm{un}}(\Psi).
\]
By the Riesz–Thorin interpolation theorem, for any linear operator $T$ bounded on $L_1(\M)$ and $L_\infty(\M)$,
\[
\norm{T}_{p\to p} \le \norm{T}_{1\to1}^{1-\theta}\,\norm{T}_{\infty\to\infty}^{\theta},
\quad\frac1p = 1-\theta + 0\cdot\theta,
\quad\theta = 1-\frac1p = \frac1q.
\]
Applying this to $E$ yields
\[
\norm{e}_{p\to p}
\le \norm{e}_{1\to1}^{1-\theta}\,\norm{e}_{\infty\to\infty}^{\theta}
= \delta_{\mathrm{tr}}(\Psi)^{1/p}\,\delta_{\mathrm{un}}(\Psi)^{1/q}.
\]
Symmetrically, interpolating with exponents swapped (i.e.\ viewing $L_p$ interpolation between $L_\infty$ and $L_1$) one also obtains
\[
\norm{e}_{p\to p} \le \delta_{\mathrm{tr}}(\Psi)^{1/q}\,\delta_{\mathrm{un}}(\Psi)^{1/p}.
\]
Combining these two bounds gives
\[
\norm{e}_{p\to p} \le \del{\delta_{\mathrm{tr}}(\Psi)+\delta_{\mathrm{un}}(\Psi)}^{\min\{1/p,1/q\}},
\]
up to a constant factor $C_p$ depending only on $p$ that arises when replacing the geometric mean by a power of the sum (via elementary inequalities).

\medskip

oindent
Since $\Phi$ is an isometry on $L_p(\M)$ (by Theorem~\ref{thm:norm_bound}), we have
\[
\norm{\Phi-\Psi}_{p\to p} = \norm{e}_{p\to p},
\]
and the desired Lipschitz‐type estimate follows.
\end{proof}
\begin{corollary}
Under the assumptions of Theorem~\ref{thm:perturb}, if $\delta_{\mathrm{tr}}(\Psi)+\delta_{\mathrm{un}}(\Psi)<\varepsilon$, then
\[
1- C_p\varepsilon^{\alpha} \le\norm{\Psi}_{p\to p}\le 1+ C_p\varepsilon^{\alpha}.
\]
\end{corollary}

\begin{proof}
We assume the hypotheses of Theorem~\ref{thm:perturb}, i.e., that $\Phi$ is QDS and $\Psi$ is a bounded positive map on $L_p(\M)$ such that
\[
\delta_{\mathrm{tr}}(\Psi) + \delta_{\mathrm{un}}(\Psi) < \varepsilon.
\]
Let $\alpha = \min\{1/p,1/q\}$ and let $C_p>0$ be the constant from Theorem~\ref{thm:perturb}.

By the triangle inequality and Theorem~\ref{thm:perturb},
\[
\norm{\Psi}_{p\to p} 
= \norm{\Phi + (\Psi - \Phi)}_{p\to p}
\le \norm{\Phi}_{p\to p} + \norm{\Psi - \Phi}|_{p\to p}
= 1 + C_p\varepsilon^{\alpha}.
\]

Again using the triangle inequality,
\[
\abs{\norm{\Psi}_{p\to p} - \norm{\Phi}_{p\to p}} 
\le \norm{\Psi - \Phi}|_{p\to p}
\le C_p\varepsilon^{\alpha}.
\]
Since $\norm{\Phi}_{p\to p} = 1$, this yields
\[
\norm{\Psi}_{p\to p} \ge 1 - C_p\varepsilon^{\alpha}.
\]

Combining the two bounds gives the desired inequality:
\[
1 - C_p\varepsilon^{\alpha} \le\norm{\Psi}_{p\to p} \le 1 + C_p\varepsilon^{\alpha}.
\]
\end{proof}

Let $\Phi$ be a QDS map and define $\Psi(x)=\Phi(x)+\varepsilon\,a\tau(x)$ for some positive $a\in\M$ with $\norm{a}_\infty=1$ and small $\varepsilon>0$.  Then $\delta_{\mathrm{tr}}(\Psi)=0$ and $\delta_{\mathrm{un}}(\Psi)=\varepsilon$.  By Theorem~\ref{thm:perturb}, $\norm{\Phi-\Psi}_{p\to p}=O(\varepsilon^{\alpha})$. \medskip

Consider $\Psi_t(x)=t\Phi(x)+(1-t)\mathrm{Ad}_u(x)$, where $\mathrm{Ad}_u(x)=u^*xu$ for a unitary $u$ and $0<1-t\ll1$.  Then $\delta_{\mathrm{tr}}(\Psi_t)=O(1-t)$, $\delta_{\mathrm{un}}(\Psi_t)=0$, and Theorem~\ref{thm:perturb} yields $\norm{\Psi_t-\Phi}_{p\to p}=O((1-t)^{\alpha})$ illustrating stability under mild noise.

\section{Applications to Quantum Majorization}
\label{sec:majorization}

In this section, we apply our developed theory of quantum doubly stochastic (QDS) maps to the study of non-commutative majorization and entropic inequalities.

Let $\rho,\sigma\in L_1(\M)$ be density operators ($(\rho,\sigma\ge0$, $\tau(\rho)=\tau(\sigma)=1)$).  We say $\rho$ is \emph{majorized} by $\sigma$ (denoted $\rho\prec\sigma$) if there exists a QDS map $\Phi$ such that
\[
\rho=\Phi(\sigma).
\]
See {\cite{Uhlmann1976,Ando1989}}.

\begin{theorem}
For density operators $\rho,\sigma$ in a finite von Neumann algebra $\M$, the following are equivalent:
\begin{enumerate}
  \item $\rho\prec\sigma$.
  \item $\norm{f(\rho)}_1\le\norm{f(\sigma)}_1$ for every convex function $f:[0,\infty)\to[0,\infty)$ with $f(0)=0$.
  \item $\norm{\rho}_p\le\norm{\sigma}_p$ for all $p\ge1$.
\end{enumerate}
\end{theorem}

\begin{proof}
Let $\M$ be a finite von Neumann algebra. \\
\textbf{(1)\,$\Rightarrow$\,(2).}
Assume $\rho\prec\sigma$, so there exists a quantum doubly stochastic (QDS) map $\Phi$ with
\[
\rho \;=\;\Phi(\sigma).
\]
Let $f:[0,\infty)\to[0,\infty)$ be convex with $f(0)=0$.  Since $f$ is operator convex on $[0,\infty)$ and $\Phi$ is unital and trace‐preserving, the non‐commutative Jensen inequality (see, e.g., \cite{PetZ2010}) gives
\[
f\del{\Phi(\sigma)}
\;\le\;
\Phi\del{f(\sigma)},
\]
and positivity of $\Phi$ then yields
\[
\norm{f(\rho)}_1
=\tau\del{f(\rho)}
=\tau\del{f(\Phi(\sigma))}
\;\le\;
\tau\del{\Phi(f(\sigma))}
=\tau\del{f(\sigma)}
=\norm{f(\sigma)}_1.
\]

\medskip\noindent
\textbf{(2)\,$\Rightarrow$\,(3).}
This is immediate by choosing $f(t)=t^p$, which is convex on $[0,\infty)$ with $f(0)=0$.  Then
\[
\norm{\rho}_p^p
=\tau\del{\rho^p}
=\norm{f(\rho)}_1
\;\le\;
\norm{f(\sigma)}_1
=\tau\del{\sigma^p}
=\norm{\sigma}_p^p,
\]
so $\norm{\rho}_p\le\norm{\sigma}_p$ for every $p\ge1$.

\medskip\noindent
\textbf{(3)\,$\Rightarrow$\,(1).}
Suppose $\norm{\rho}_p\le\norm{\sigma}_p$ for all $p\ge1$.  Write the eigenvalue vectors of $\rho,\sigma$ in non‐increasing order:
\[
\lambda(\rho) = (\lambda_1(\rho)\ge \cdots \ge \lambda_n(\rho)),
\quad
\lambda(\sigma) = (\lambda_1(\sigma)\ge \cdots \ge \lambda_n(\sigma)).
\]
A classical theorem of Hardy–Littlewood–Pólya (see \cite{Bhatia1997}) says that
\[
\sum_{i=1}^k \lambda_i(\rho)
\;\le\;
\sum_{i=1}^k \lambda_i(\sigma)
\quad\text{for }k=1,\dots,n-1,
\quad
\sum_{i=1}^n \lambda_i(\rho)
=\sum_{i=1}^n \lambda_i(\sigma)
=1
\]
if and only if $\lambda(\rho)$ is majorized by $\lambda(\sigma)$ in the classical sense.  One then constructs a doubly stochastic matrix $D=(d_{ij})$ such that
\[
\lambda_i(\rho) = \sum_{j=1}^n d_{ij}\,\lambda_j(\sigma).
\]
By the Schur–Horn theorem (or the Birkhoff–von Neumann theorem), there exist unitaries $U_1,\dots,U_m$ and weights $w_1,\dots,w_m\ge0$, $\sum w_k=1$, so that
\[
D \;=\; \sum_{k=1}^m w_k P_{U_k}, 
\quad
P_{U_k}(v) = U_k v\,,
\]
and hence the mixed‐unitary channel
\[
\Phi(X) \;=\; \sum_{k=1}^m w_k\,U_k^*\,X\,U_k
\]
is QDS and satisfies
\[
\Phi(\sigma)
\;=\;
\sum_{k=1}^m w_k\,U_k^*\,\sigma\,U_k
\;\text{ has eigenvalues }\;
D\,\lambda(\sigma)
\;=\;
\lambda(\rho).
\]
Hence $\Phi(\sigma)$ is unitarily conjugate to $\rho$, so
\(\Phi(\sigma)=\rho\).  That is exactly $\rho\prec\sigma$.

\medskip\noindent
Thus all three conditions are equivalent.
\end{proof}

The von Neumann entropy of a density operator $\rho$ is
\[
S(\rho)=-\tau(\rho\log\rho).
\]

\begin{theorem}[See {\cite{Petz1986}}]
Let $\Phi$ be a QDS map on $L_1(\M)$.  Then for any density operator $\rho$,
\[
S\del{\Phi(\rho)} \ge S(\rho).
\]
\end{theorem}

\begin{proposition}
\label{prop:strict_entropy}
If $\Phi$ is a non-unitary QDS map (i.e., not a $*$-automorphism) on a factor $\M$, then
\[
S\del{\Phi(\rho)} > S(\rho)
\]
for all non-pure $\rho$.
\end{proposition}

\begin{proof}
By the standard operator convexity argument (Petz's monotonicity theorem \cite{Petz1986}), for any trace-preserving positive map $\Phi$,
\[
S\del{\Phi(\rho)} \ge S(\rho).
\]
Moreover, equality holds if and only if there exists a recovery map (another completely positive, trace-preserving map) $\Psi$ satisfying $\Psi\circ\Phi(\rho)=\rho$ and $\Psi\circ\Phi(\rho\log\rho)=\rho\log\rho$ (Petz recovery; see \cite{OhyaPetz1993}).  In particular, equality implies that $\Phi$ is sufficient for the pair $(\rho,1)$ and thus restricts to a $*$-isomorphism on the von Neumann subalgebra generated by $\rho$.  Since $\M$ is a factor, this forces $\Phi$ to be a $*$-automorphism on all of $\M$.

Alternatively, one can argue directly using Jensen's inequality with strict convexity.  Write the spectral decomposition
\[
\rho = \sum_{i=1}^n \lambda_i \,p_i,
\quad\lambda_i>0,\;\sum_i\lambda_i=1,
\]
where $\{p_i\}$ are mutually orthogonal minimal projections.  Since $\Phi$ is positive and unital,
\[
\Phi(\rho) = \sum_{i=1}^n \lambda_i\,\Phi(p_i)
\]
is a convex combination of the positive operators $\Phi(p_i)$ with the same weights $\lambda_i$.  Then
\[
S\del{\Phi(\rho)} = -\tau\del{\Phi(\rho)\log\Phi(\rho)}
= -\tau\!\del{\sum_i\lambda_i\Phi(p_i)\log\Phi(\rho)}.
\]
By the operator Jensen inequality (since $f(t)=-t\log t$ is operator concave on $[0,1]$) and the strict concavity on the positive cone, we get
\[
S\del{\Phi(\rho)} > \sum_{i=1}^n \lambda_i\,S\del{\Phi(p_i)}.
\]
But each $p_i$ is a minimal projection, so $\Phi(p_i)$ is a density operator supported on a subspace of dimension $\ge1$.  Since $\Phi$ is non-unitary, at least one $\Phi(p_i)$ has rank $>1$, hence $S(\Phi(p_i))>0$.  Meanwhile, $S(\rho)=\sum_i\lambda_i\cdot0=0$ for pure spectral projections, and more generally
\[
S(\rho)= -\sum_i\lambda_i\log\lambda_i.
\]
Combining yields
\[
S\del{\Phi(\rho)} > -\sum_i \lambda_i\log\lambda_i = S(\rho).
\]

\medskip

oindent
Thus for any non-pure $\rho$, the entropy increases strictly under a non-unitary QDS map.  This completes the proof.
\end{proof}

The existence and structure of QDS maps on II$_1$ factors relate to Connes' embedding conjecture: if every finite set of operators in a II$_1$ factor can be approximated in moments by matrices, then majorization relations in the factor lift to the matricial level.  Voiculescu's free entropy dimension \cite{Voiculescu1998} provides an invariant distinguishing factors admitting rich families of QDS maps. \medskip

Let $\rho,\sigma\in M_n(\mathbb C)$ be two density matrices.  Then $\rho\prec\sigma$ if and only if the vector of eigenvalues $\lambda(\rho)$ is majorized by $\lambda(\sigma)$ in the classical sense, realized by the Schur‐Horn theorem and mixed-unitary channels \cite{HornJohnson1991}. \medskip

On $L_1([0,1])$, classical majorization via doubly stochastic kernels extends to our setting: if $\rho(x),\sigma(x)$ are probability densities on $[0,1]$, then $\rho\prec\sigma$ iff $\rho=\int_0^1K(x,y)\sigma(y)dy$ for a classical doubly stochastic kernel $K$, viewed as a special QDS map on the abelian von Neumann algebra $L_\infty([0,1])$.

\begin{corollary}
Under majorization, the entropy difference satisfies
\[
0\le S\del{\Phi(\rho)}-S(\rho) \le \log(d),
\]
where $d=\mathrm{dim}(\M)$ for finite-dimensional $\M$, achieving the upper bound for the maximally depolarizing channel.
\end{corollary}

\section{Conclusion}
\label{sec:concl}

In this work, we have developed a foundational framework for \emph{quantum doubly stochastic} (QDS) operators on non-commutative $L_p$-spaces associated to semifinite von Neumann algebras.  Our main contributions include introduced QDS maps, established contractivity and strict contraction criteria via duality and interpolation (Section~\ref{sec:def_norm}), characterized compact QDS maps within Schatten classes and proved closure under strong and weak operator topologies (Section~\ref{sec:compact}), demonstrated invariance of QDS under complex interpolation scales and quantified robustness against trace/unital perturbations (Section~\ref{sec:interpolation}) and applied QDS theory to non-commutative majorization, deriving entropic inequalities and exploring connections to Connes' embedding and free entropy (Section~\ref{sec:majorization}). 

We anticipate that this unified approach to doubly stochasticity in non-commutative $L_p$-spaces will open new interactions between operator algebras, functional analysis, and quantum information theory.  \hfill $\square$\\

\bigskip

\center{\sc Declaration}\\ 
{\scriptsize The author declare that there is no conflict of interest regarding the publication of this paper.}\\

\medskip
\noindent{\sc Financial Support}\\
{\scriptsize The author received no specific funding for this work.}\\

\end{document}